\documentclass[reqno] {amsart}
\usepackage{amsmath}
\usepackage{amsthm}
\usepackage{amssymb}
\usepackage{latexsym}
\usepackage{enumerate}

\newtheorem{thm}{Theorem}[section]
\newtheorem{prop}{Proposition}[section]
\newtheorem{lem}{Lemma}[section]

\title{The radial curvature of an end that makes eigenvalues vanish in the essential spectrum II}

\author{Hironori Kumura}

\date{}

\begin{document}

\maketitle

\begin{abstract}
Under the quadratic-decay-conditions of the radial curvatures of an end, we shall derive growth estimates of solutions to the eigenvalue equation and show the absence of eigenvalues. 
\end{abstract}
\section{Introduction}

The Laplace-Beltrami operator $\Delta$ on a noncompact complete Riemannian manifold $M$ is essentially self-adjoit on $C^{\infty}_0(M)$ and its self-adjoit extension to $L^2(M)$ has been studied by several authors from various points of view. 
Especially, the problem of the absence of eigenvalues was 
discussed in $[2,3,4,5,6,7,9,10,13,15,$\linebreak
$17, 19]$. 
The purpose of this paper is to clarify the relationship between the curvatures at infinity and the spectral structure of the Laplacian. 
In particular, this paper discusses growth estimates of solutions to the eigenvalue equation and show the absence of eigenvalues of the Laplacian. 

Let us now look at the previous works which treat the case the curvature $K$ tend to zero at infinity. 
They mainly studied the case that $(M,g)$ is simply connected and complete and that $K$ is nonpositive. 
We shall recall decay conditions on $K$ which ensure the absence of positive eigenvalues. 
In this case, the earlier works mainly studied the case that $\mathrm{dim~} M = 2$, because their arguments require faster than quadratic decay for $K$ which, in dimension greater than two, would force $(M,g)$ to be isometric with ${\bf R}^n$ (see Green and Wu \cite{G-W}). 
That is why this problem for higher dimensions remains a challenge so far. 
For example, it was assumed in \cite{D3} that $\int_1^{\infty} r^{\beta_1} |K| \,dr < \infty$ and $\lim_{r\to \infty} r^{\beta_2} K = 0$, where $\beta_1\ge 2$ and $\beta_2\ge 3$ are constants. 
Roughly speaking, this curvature condition is $K = O(r^{-3-\varepsilon})$. 
In this paper, we shall treat manifolds of all dimensions under the assumption of some quadratic decay for the curvature, and prove the absence of positive eigenvalues. 
We note here that Escobar and Freire \cite{E-F} studied the nonnegative curvature case. 
However, their arguments require {\it global} curvature conditions on $M$. 

We shall state our results more precisely. 
Let $(M,g)$ be an $n$-dimensional noncompact complete Riemannian manifold and $U$ an open subset of $M$. 
We shall say that $M-U$ is an {\it end with radial coordinates} if and only if the boundary $\partial U$ is compact, connected, and $C^{\infty}$ and the outward normal exponential map $\exp_{\partial U}^{\perp} : N^{+}(\partial U) \to M - \overline{U}$ induces a diffeomorphism, where $N^{+}(\partial U) = \{ v\in T(\partial U) \mid v {\rm ~is~outward~normal~to~}\partial N \}$. 
Note that $U$ is not necessarily relatively compact. 
Let $r$ denote the distance function from $\partial U$ defined on the end $M-U$. 
We shall say that a $2$-plane $\pi \subset T_xM$ $(x\in M-U)$ is {\it radial} if $\pi$ contains $\nabla r$, and, by the {\it radial curvature}, we mean the restriction of the sectional curvature to all the radial planes. 
In the sequel, the following notations will be used:
\begin{align*}
   & B(s,t) = \{ x\in M - \overline{U} \mid s < r(x) < t \} \quad 
              \mathrm{for}~~0 \le s < t ;\\
   & B(s,\infty) = \{ x\in M - \overline{U} \mid s<r(x) \} \quad 
              \mathrm{for}~~0 \le s <\infty;\\
   & S(t) = \{ x\in M - \overline{U} \mid r(x)=t \} \quad 
              \mathrm{for}~~0 \le t < \infty ;\\
   & \sigma(-\Delta) =  \,{\rm the~spectrum~of}~-\Delta;\\
   & \sigma_{\rm p}(-\Delta) =  \,{\rm the~set~of~all~eigenvalues~of}~
     - \Delta;\\
   & \sigma_{\rm ess}(-\Delta) = \,{\rm the~essential~spectrum~of}~- \Delta;\\
   & K_{{\rm rad.}} =  \,{\rm the~radial~curvature~on~}M-U.
\end{align*}
Moreover, we denote the Riemannian measure of $(M,g)$ by $dv_g$, and the induced measures from $dv_g$ on each $S(t)~( t > 0 )$ simply by $dA$. 

We shall consider the eigenvalue equation 
\begin{align*}
     \Delta u + \lambda u = 0 \qquad {\rm on}~~E = M - \overline{U}
\end{align*}
and drive a growth estimate at infinity of solutions $u$, from which will follow the absence of eigenvalues; 
\begin{thm}
Let $(M,g)$ be an $n$-dimensional complete Riemannian manifold and $U$ an open subset of $M$. 
Assume that $E:= M - \overline{U}$ is an end with radial coordinates and set $r = {\rm dist }(U,*)$. 
We assume that there exists constants $r_0 >0$ such that the second fundamental forms $\nabla dr$ of the level hypersurfaces $\{S(r)\}_{r\ge r_0}$ satisfies 
\begin{align}
   \frac{a}{r} ( g - dr \otimes dr ) 
   \le \nabla dr \le 
   \frac{b}{r} ( g - dr \otimes dr ) 
   \qquad {\rm on}~~B(r_0,\infty), \tag{$*_1$} 
\end{align}
where $a >0$ and $b>0$ are constants satisfying
\begin{align}
   a \le b \quad {\rm and} \quad \frac{n+1}{n-1}a > b . \tag{$*_2$}
\end{align}
In addition, we assume that
\begin{align}
   {\rm Ric}\, (\nabla r, \nabla r) \ge - (n-1)\frac{b_1(r)}{r}
   \qquad {\rm on}~~B(r_0,\infty), \tag{$*_3$}
\end{align}
where $ b_1 (t) $ is a positive-valued decreasing function of $t \in [r_0,\infty)$ satisfying $\lim_{t \to \infty} b_1 (t) = 0$. 
Let $\lambda > 0$ be a constant and $u$ a nontrivial solution to the equation
\begin{align*}
     \Delta u + \lambda u = 0 \qquad {\rm on}~~B(r_0,\infty).
\end{align*}
Then, for any 
\begin{align}
   \gamma > \frac{n-1}{2} ( b - a ), \tag{$*_4$}
\end{align}
we have
\begin{align}
   \liminf_{t \to \infty}~t^{\gamma} \int_{S(t)} 
   \left\{ \left( \frac{\partial u}{\partial r} \right)^2 + |u|^2 \right\}
   \,dA \neq 0 .
\end{align}
\end{thm}
\begin{thm}
Let $(M,g)$ be an $n$-dimensional complete Riemannian manifold and $U$ an open subset of $M$. 
Assume that $E:= M - \overline{U}$ is an end with radial coordinates and set $r = {\rm dist }(U,*)$. 
We assume that there exists constants $r_0 >0$ such that the second fundamental forms $\nabla dr$ of the level hypersurfaces $\{S(r)\}_{r\ge r_0}$ satisfies 
\begin{align}
   \frac{a}{r} ( g - dr \otimes dr ) 
   \le \nabla dr \le 
   \frac{b}{r} ( g - dr \otimes dr ) 
   \qquad {\rm on}~~B(r_0,\infty), \tag{$*_1$} 
\end{align}
where $a >0$ and $b>0$ are constants satisfying
\begin{align}
   a \le b \quad {\rm and} \quad \frac{n+1}{n-1}a > b . \tag{$*_2$}
\end{align}
In addition, we assume that
\begin{align}
   {\rm Ric}\, (\nabla r, \nabla r) \ge - (n-1)\frac{b_1(r)}{r}
   \qquad {\rm on}~~B(r_0,\infty), \tag{$*_3$}
\end{align}
where $ b_1 (t) $ is a positive-valued decreasing function of $t \in [r_0,\infty)$ satisfying $\lim_{t \to \infty} b_1 (t) = 0$. 
Furthermore, if we add the assumption
\begin{align}
  1 > \frac{n-1}{2} ( b - a )   \tag{$*_5$},
\end{align}
to ones above, then $\sigma (-\Delta) = [0,\infty)$ and $\sigma _{{\rm p}}( -\Delta ) = \emptyset$. 
\end{thm}
Theorem $1.1$, $1.2$ and comparison theorem in Riemannian geometry (Kasue \cite{Kasue}) yield the following:
\begin{thm}
Let $(M,g)$ be an $n$-dimensional complete Riemannian manifold and $U$ an open subset of $M$. 
Assume that $E:= M - \overline{U}$ is an end with radial coordinates. 
We assume that there exists a constant $r_0 > 0$ such that 
\begin{align}
   & \nabla dr \bigl|_{TS(r_0)\times TS(r_0)}
     \begin{cases}
        \ge \displaystyle \frac{a}{r_0}  
        & \quad \text{if }~~0 < a < 1, \vspace{1mm} \\
        > 0
        & \quad \text{if }~~a \ge 1; 
     \end{cases}
     \tag{$*_6$} \vspace{2mm} \\
   & \nabla dr \bigl|_{TS(r_0)\times TS(r_0)} 
     \le \frac{b}{r_0} \hspace{10mm} \text{if }~~0 < b < 1 ;\tag{$*_7$} \\
   & - \frac{b(b-1)}{r^2}\le K_{{\rm rad.}} \le \frac{a(1-a)}{r^2} \qquad 
     \mathrm{on}~B(r_0,\infty), \tag{$*_8$}
\end{align}
where $ a > 0$ and $ b > 0$ are constants satisfying 
\begin{align}
  a \le b \quad {\rm and} \quad \frac{n+1}{n-1}a > b. \tag{$*_2$} 
\end{align}
Let $\lambda>0$ be a constant and $u$ a nontrivial solution to 
\begin{align*}
     \Delta u + \lambda u = 0 \qquad {\rm on}~~B(r_0,\infty).
\end{align*}
Then, we have 
\begin{align*}
     \liminf_{t\to \infty} ~t^{\gamma} \int_{S(t)}
     \left\{ \left( \frac{\partial u}{\partial r} \right)^2 + |u|^2 
     \right\} \,dA \neq 0 
     \quad {\rm for~any}~~\gamma > \frac{n-1}{2} ( b - a ) .
\end{align*}
\end{thm}
\begin{thm}
Let $(M,g)$ be an $n$-dimensional complete Riemannian manifold and $U$ an open subset of $M$. 
Assume that $E:= M - \overline{U}$ is an end with radial coordinates. 
We assume that there exists a constant $r_0 > 0$ such that 
\begin{align}
   & \nabla dr \bigl|_{TS(r_0)\times TS(r_0)}
     \begin{cases}
        \ge \displaystyle \frac{a}{r_0}  
        & \quad \text{if }~~0 < a < 1, \vspace{1mm} \\
        > 0
        & \quad \text{if }~~a \ge 1; 
     \end{cases}
     \tag{$*_6$} \vspace{2mm} \\
   & \nabla dr \bigl|_{TS(r_0)\times TS(r_0)} 
     \le \frac{b}{r_0} \hspace{10mm} \text{if }~~0 < b < 1 ;\tag{$*_7$} \\
   & - \frac{b(b-1)}{r^2}\le K_{{\rm rad.}} \le \frac{a(1-a)}{r^2} \qquad 
     \mathrm{on}~B(r_0,\infty), \tag{$*_8$}
\end{align}
where $ a > 0$ and $ b > 0$ are constants satisfying 
\begin{align}
  a \le b \quad {\rm and} \quad \frac{n+1}{n-1}a > b. \tag{$*_2$} 
\end{align}
Furthermore, if we add the assumption
\begin{align}
  1 > \frac{n-1}{2} ( b - a )   \tag{$*_5$},
\end{align}
to ones above, then $\sigma (-\Delta) = [0,\infty)$ and $\sigma _{{\rm p}}( -\Delta ) = \emptyset$. 
\end{thm}
In Theorem $1.3$ and $1.4$, note that we do not assume an explicit upper bound of $\nabla dr \bigl|_{TS(r_0)\times TS(r_0)}$ if $ b \ge 1$. \vspace{2mm}

Our method is a modification of solutions of Koto \cite{Kato}, Eidus \cite{E}, Roze \cite{R}, and Mochizuki \cite{M} to the analogous problem for the Schr\"odinger equation on Euclidian space. 

\vspace{3mm}
The author would like to express his gratitude to Professor Minoru Murata. 
He kindly informed the author of several facts about the analogous results for the Schr\"odinger equation on Euclidian space. 

\section{The second fundamental form and radial curvature}

In this section, we shall confirm our geometric situation. 
On an end with radial coordinates, the second fundamental forms $\nabla dr$ of the level hypersurfaces $\{ S(r) \}_{r\ge 0}$ describes the metric growth of the surfaces $\{ S(r) \}_{r\ge 0}$ and the radial curvatures controls the second fundamental forms $\nabla dr$; the comparison theorem in Riemannian geometry ( Kasue \cite{Kasue} ) yields the following propositions: 
\begin{prop}
Let $(M,g)$ be an $n$-dimensional complete Riemannian manifold and $U$ an open subset of $M$. 
We assume that $E:= M - U$ has radial coordinates. 
Let $r_0 > 0$ and $\alpha$ be constants. 
We assume that
\begin{align*}
  \alpha \in (0,1).
\end{align*}
Then, the following holds$:$
\begin{enumerate}[$(1)$]
  \item If 
	\begin{align*}
      \begin{cases}
        \displaystyle \nabla dr 
        \ge \frac{\alpha}{r_0} \left( g - dr \otimes dr \right)  
        & \text{on}~~S(r_0), \vspace{2mm} \\ 
        \displaystyle K_{{\rm rad.}} \le \frac{\alpha(1-\alpha)}{r^2} 
        & \text{on }~~B(r_0, \infty),
      \end{cases}
	\end{align*}
    then we have
	\begin{align*}
	   \nabla dr & \ge \frac{\alpha}{r} \left( g - dr \otimes dr \right) 
	      \qquad {\rm on}~~B(r_0, \infty) .
	\end{align*}
  \item If 
     \begin{align*}
      \begin{cases}
        \displaystyle  \nabla dr 
        \le \frac{\alpha}{r_0} \left( g - dr \otimes dr \right)  
        & \text{on}~~S(r_0), \vspace{2mm} \\
        \displaystyle  K_{{\rm rad.}} \ge \frac{\alpha(1-\alpha)}{r^2} 
        & \text{on }~~B(r_0, \infty),
      \end{cases}
     \end{align*}
    then we have
     \begin{align*}
	   \nabla dr & \le \frac{\alpha}{r} \left( g - dr \otimes dr \right) 
	      \qquad {\rm on}~~B(r_0, \infty) .
     \end{align*}
\end{enumerate}
\end{prop}
\begin{prop}
Let $(M,g)$ be an $n$-dimensional complete Riemannian manifold and $U$ an open subset of $M$. 
We assume that $E:= M - U$ has radial coordinates. 
Let $r_0 > 0$, $\alpha$, and $\beta > 0$ be constants. 
We assume that
\begin{align*}
  \alpha \ge 1.
\end{align*}
Then, the following holds$:$
\begin{enumerate}[$(1)$]
  \item If 
     \begin{align*}
      \begin{cases}
        \displaystyle \nabla dr 
        \ge \beta \left( g - dr \otimes dr \right)  
        & \text{on}~~S(r_0), \vspace{2mm} \\
        \displaystyle K_{{\rm rad.}} \le - \frac{ \alpha ( \alpha - 1 )}{r^2} 
        & \text{on }~~B(r_0, \infty),
      \end{cases}
     \end{align*}
    then we have
     \begin{align*}
	   \nabla dr & 
	   \ge \frac{\alpha}{r - r_0 + \frac{\alpha}{\beta}} 
	   \left( g - dr \otimes dr \right) 
	      \qquad {\rm on}~~B(r_0, \infty) .
     \end{align*}
  \item If 
     \begin{align*}
      \begin{cases}
        \displaystyle \nabla dr 
        \le \beta \left( g - dr \otimes dr \right)  
        & \text{on}~~S(r_0), \vspace{2mm} \\
        \displaystyle K_{{\rm rad.}} \ge - \frac{ \alpha ( \alpha - 1 )}{r^2} 
        & \text{on }~~B(r_0, \infty),
      \end{cases}
     \end{align*}
    then we have
     \begin{align*}
	   \nabla dr & 
	   \le \frac{\alpha}{r - r_0 + \frac{\alpha}{\beta}} 
	   \left( g - dr \otimes dr \right) 
	      \qquad {\rm on}~~B(r_0, \infty) .
     \end{align*}
\end{enumerate}
\end{prop}
For the proof of Proposition $2.1$ and $2.2$, see Kasue \cite{Kasue}.
\section{Analytic propositions} 

In this section, we shall prepare some analytic propositions toward the proof of Theorem $1.1$. 

Let $(M,g)$ be an $n$-dimensional complete Riemannian manifold and $U$ an open subset of $M$. 
Assume that $E:= M - \overline{U}$ is an end with radial coordinates. 
We shall consider the eigenvalue equation 
\begin{align*}
    \Delta u + \lambda u = 0 \qquad {\rm on}~~E:= M - \overline{U},
\end{align*}
where $\lambda >0$ is a constant.  

Let $\rho (r)$ be a $C^{\infty}$ function of $r \in [r_0,\infty)$, and put 
\begin{align*}
  v(x) = \exp \bigl( \rho (r(x)) \bigr) u(x) \qquad {\rm for}~~x \in E.
\end{align*}
Then it follows that $v$ satisfies on $B(r_0,\infty)$ the equation 
\begin{align*}
    & \Delta v - 2 \rho '(r) \frac{\partial v}{\partial r} + qv = 0,\\
    & q = |\nabla \rho (r)|^2 - \Delta \rho (r) + \lambda \nonumber \\
    & \hspace{1.8mm}  = \left| \rho '(r) \right|^2 - \rho ''(r) 
      - \rho '(r)\Delta r + \lambda.
\end{align*}
As is mentioned in section $1$, we denote by $dA$ the measures on each level surface $S(t)~~(t>0)$ induced from the Riemannian measure $dv_g$ on $(M,g)$. 
\begin{prop}
For any $\psi \in C^{\infty}(M-\overline{U})$ and $r_0\le s< t$, we have
\begin{align*}
  & \int_{B(s,t)} \left\{ |\nabla v|^2 - q|v|^2 \right\}\psi \,dv_g \\
 =& \left( \int_{S(t)} - \int_{S(s)} \right)
    \left( \frac{\partial v}{\partial r} \right) \psi v \,dA 
    - \int_{B(s,t)}
    \left\langle 
    \nabla \psi + 2 \psi \rho '(r) \nabla r , \nabla v 
    \right\rangle v \,dv_g.
\end{align*}
\end{prop}
\begin{prop}
Let $\nabla r, X_1, X_2, \cdots , X_{n-1}$ be an orthonormal base for the tangent space $T_xM$ at each point $x \in M - \overline{U}$. 
Then, for any real numbers $\gamma$, $\varepsilon$, and $0\le s<t$, we have
\begin{align*}
  & \int_{S(t)} r^{\gamma} \left\{ 
    \left(\frac{\partial v}{\partial r}\right)^2 + \frac{1}{2}q|v|^2 
    - \frac{1}{2}|\nabla v|^2 
    + \frac{\gamma - \varepsilon }{2r}
    \frac{\partial v}{\partial r}v \right\} \,dA \nonumber \\
  & + \int_{S(s)} r^{\gamma} \left\{ \frac{1}{2}|\nabla v|^2
    - \frac{1}{2}q|v|^2 - \left(\frac{\partial v}{\partial r}\right)^2
    - \frac{\gamma-\varepsilon }{2r}\frac{\partial v}{\partial r}v \right\} 
    \,dA \nonumber \\
 =& \int_{B(s,t)} r^{\gamma -1} \left\{ r(\nabla dr)(\nabla v,\nabla v)
    - \frac{1}{2}(r\Delta r + \varepsilon )
    \sum_{i=1}^{n-1} \left( dv(X_i) \right)^2 \right\} \,dv_g \\
  & + \int_{B(s,t)} r^{\gamma-1} \left\{ \gamma - 
    \frac{1}{2}( r\Delta r + \varepsilon ) + 2r \rho'(r) \right\}
    \left(\frac{\partial v}{\partial r}\right)^2 \,dv_g \nonumber \\
  & + \frac{1}{2} \int_{B(s,t)} r^{\gamma-1}
    \left\{ r\left( \frac{\partial q}{\partial r} \right)
    + q( r\Delta r + \varepsilon ) \right\} |v|^2 \,dv_g \nonumber \\
  & + \frac{\gamma -\varepsilon}{2} \int_{B(s,t)} r^{\gamma-1}
    \left\{ \frac{\gamma-1}{r} + 2\rho ' (r) \right\} 
    \frac{\partial v}{\partial r}v \,dv_g.
\end{align*}
\end{prop}
%
%
\begin{lem}
For any $ \beta \in {\bf R}$, we have 
\begin{align}
   \left( \int_{S(t)} - \int_{S(s)} \right) r^{\beta } |v|^2 \,dA
 = \int_{B(s,t)} r^{\beta }
   \left\{ \left( \Delta r + \frac{\beta }{r} \right)|v|^2
   + 2v\frac{\partial v}{\partial r} \right\} \,dv_g .
\end{align}
\end{lem}
\begin{lem}
For any $ m \in {\bf R}$, we have
\begin{align*}
   &  \int_{B(x,\infty)} r^{1-2m} 
      \left\{ |\nabla v|^2 - q|v|^2 \right\} \,dv_g  \nonumber \\
 = &  - \frac{1}{2} \frac{d}{dx} \left( x^{1-2m} \int_{S(x)} |v|^2 \,dA \right)
      - \frac{1}{2} \int_{S(x)} r^{-2m} 
      \left\{ 2m - 1 - r \Delta r \right\} |v|^2 \,dA
      \nonumber \\
   &  - \int_{B(x,\infty)} r^{-2m} 
      \left( \frac{\partial v}{\partial r} \right) v \,dv_g,
\end{align*}
\end{lem}
Proposition $3.1$ and $3.2$ are obtained by setting $c=0$ in [18, Proposition $3.1$ and $3.3$], respectively; 
Lemma $3.1$ also follows from [18, Lemma $3.2$]; 
Lemma $3.2$ is also got by putting $c=0$ in [18, equations $(32)$ and $(33)$]. 
\section{Faster than polynomial decay}

Let $(M,g)$ be an $n$-dimensional noncompact complete Riemannian manifold and $U$ an open subset of $M$. 
We assume that $E: = M - \overline{U}$ is an end with radial coordinates. 
We denote $r = {\rm dist}(U,*)$ on $E$. 
Let us set 
\begin{align*}
   \widetilde{g} = g - dr \otimes dr
\end{align*} 
and assume that there exists $r_0>0$ such that 
\begin{align}
   \frac{a}{r} \,\widetilde{g} \le \nabla dr \le \frac{b}{r} \,\widetilde{g} 
   \qquad {\rm on}~~B(r_0,\infty), \tag{$*_1$}
\end{align}
where $a$ and $b$ are positive constants satisfying 
\begin{align}
   a \le b \quad {\rm and} \quad \frac{n+1}{n-1}a > b  \tag{$*_2$} .
\end{align}
We also assume that 
\begin{align}
   {\rm Ric}\, (\nabla r, \nabla r) \ge - (n-1)\frac{b_1(r)}{r}
   \qquad {\rm on}~~B(r_0,\infty), \tag{$*_3$}
\end{align}
where $ b_1 (t) $ is a positive-valued decreasing function of $t \in [r_0,\infty)$ satisfying \linebreak 
$\lim_{t \to \infty} b_1 (t) = 0$. 

In the sequel, we shall often use the following notation for simplicity:
\begin{align*}
   \widehat{a} = ( n - 1 ) a;~
   \widehat{b} = ( n - 1 ) b;~
   \widehat{b}_1 (r) = ( n - 1 ) b_1 (r).
\end{align*} 
\begin{prop}
Assume that there exist constants $r_0 > 0$, $a$, and $b$ such that $(*_1)$, $(*_2)$, and $(*_3)$ hold. 
Let $\lambda > 0$ be a constant and $u$ a solution to 
\begin{align*}
     \Delta u + \lambda u = 0 \qquad {\rm on}~~B(r_0,\infty).
\end{align*}
Moreover, let 
\begin{align}
   \gamma > \frac{n-1}{2} ( b - a )   \tag{$*_4$}
\end{align} 
be a constant and assume that $u$ satisfies the condition$:$
\begin{align}
   \liminf_{t\to \infty} ~ t^{\gamma} \int_{S(t)}
   \left\{ \left( \frac{\partial u}{\partial r} \right)^2 + |u|^2 
   \right\} \,dA=0.
\end{align}
Then, we have for any $m > 0$
\begin{align}
   \int_{B(r_0,\infty)} r^m \left\{ |u|^2 + |\nabla u|^2 \right\} 
    \,dv_g < \infty.
\end{align}
\end{prop}
\begin{proof}
We shall combine Proposition $3.2$ and Lemma $3.1$; put $ \rho (r) = 0 $ in Proposition $3.2$; put $\beta = \gamma -1$ in Lemma $3.1$ and multiply $(2)$ by a positive constant $\alpha$. 
Then $v = u$ and $q = \lambda $ and 
\begin{align}
    & \int_{S(t)} r^{ \gamma }
      \left\{ \left( \frac{\partial u}{\partial r} \right)^2 
      + \frac{1}{2}\lambda |u|^2
      + \frac{\gamma - \varepsilon }{2r} \frac{\partial u}{\partial r} u
      + \frac{\alpha }{r}|u|^2 
      \right\} \,dA  \nonumber \\
    & + \int_{S(s)} r^{ \gamma }
      \left\{ \frac{1}{2}|\nabla u|^2 - \frac{1}{2}\lambda |u|^2 
      - \left( \frac{\partial u}{\partial r} \right)^2 
      - \frac{\gamma - \varepsilon }{2r}\frac{\partial u}{\partial r}u 
      - \frac{\alpha }{r}|u|^2 
      \right\} \,dA   \nonumber \\
\ge & \int_{B(s,t)} r^{ \gamma - 1} \left\{r(\nabla dr)(\nabla u,\nabla u)
      - \frac{1}{2}( r\Delta r + \varepsilon )
      \sum_{i=1}^{n-1} \left( du(X_i) \right)^2 \right\} \,dv_g \nonumber \\
    & + \int_{B(s,t)} r^{\gamma-1}
      \left\{ \gamma - \frac{1}{2}(r\Delta r + \varepsilon ) \right\}
      \left( \frac{\partial u}{\partial r} \right)^2 \,dv_g   \nonumber \\
    & + \frac{\lambda }{2}\int_{B(s,t)} r^{\gamma-1} 
      \left\{ r\Delta r + \varepsilon 
      + \frac{2\alpha}{\lambda} \left( \Delta r + \frac{\gamma -1}{r} \right)
      \right\} |u|^2 \,dv_g    \nonumber \\
    & + \int_{B(s,t)} r^{\gamma-1} 
      \left\{
      \frac{( \gamma - \varepsilon )( \gamma -1 )}{2r} + 2\alpha
      \right\}
      \frac{\partial u}{\partial r}u \,dv_g.
\end{align}
Now, our assumptions $ 2 \gamma > ( \widehat{b} - \widehat{a} )$ and $ \frac{n+1}{n-1}a > b $ respectively imply that $ 2 \gamma - \widehat{b} + \widehat{a} > 0 $ and $ \widehat{a} > \widehat{b} - 2a $, and hence, we can choose a constant $ \varepsilon $ so that 
\begin{align}
     & \widehat{a} > - \varepsilon > \widehat{b} - 2a ~(\ge 0), \\
     & 2 \gamma - \widehat{b} - \varepsilon > 0. \nonumber
\end{align}
Then, we see that for sufficiently large $r$
\begin{align*}
    &  r(\nabla dr)( \nabla u , \nabla u ) 
       - \frac{1}{2}( r\Delta r + \varepsilon ) \sum_{i=1}^{n-1} 
       \left( du(X_i) \right)^2  \nonumber  \\
    &  \hspace{20mm} \ge 
       \frac{1}{2} \bigl\{ 2a - \widehat{b} - \varepsilon \bigr\} 
       \sum_{i=1}^{n-1} \left( du(X_i) \right)^2 \ge 0;  \\
    &  \gamma - \frac{1}{2}( r \Delta r + \varepsilon )
       \ge \frac{1}{2}\{ 2 \gamma - \widehat{b} - \varepsilon \} > 0;  \\
    &  r\Delta r + \varepsilon \ge \widehat{a} + \varepsilon > 0.
\end{align*}
Therefore, bearing Shwarz inequality in mind, if we take $ \alpha > 0 $ sufficiently small and if $r_1>0$ sufficiently large, we see that for any $ t > s \ge r_1$ the right hand side of $(5)$ is bounded from below by 
\begin{align*}
    c_1 \int_{B(s,t)} r^{\gamma-1} \bigl\{ |\nabla u|^2 + |u|^2 \bigr\} \,dv_g,
\end{align*}
where $c_1=c_1(a,b,n,\lambda ,\gamma, \varepsilon, \alpha )>0$ is a constant depending only on $a,b,n,\lambda ,\gamma$, $\varepsilon$, and $\alpha$; that is, 
\begin{align}
    & \int_{S(t)} r^{ \gamma }
      \left\{ \left( \frac{\partial u}{\partial r} \right)^2 
      + \frac{1}{2}\lambda |u|^2
      + \frac{\gamma - \varepsilon }{2r} \frac{\partial u}{\partial r} u
      + \frac{\alpha }{r}|u|^2 
      \right\} \,dA  \nonumber \\
    & + \int_{S(s)} r^{ \gamma }
      \left\{ \frac{1}{2}|\nabla u|^2 - \frac{1}{2}\lambda |u|^2 
      - \left( \frac{\partial u}{\partial r} \right)^2 
      - \frac{\gamma - \varepsilon }{2r}\frac{\partial u}{\partial r}u 
      - \frac{\alpha }{r}|u|^2 
      \right\} \,dA   \nonumber \\
\ge & c_1 \int_{B(s,t)} r^{\gamma-1} 
      \bigl\{ |\nabla u|^2 + |u|^2 \bigr\} \,dv_g. 
\end{align}
Besides, by Shwarz inequality, for $r \ge r_2 := \max\{ r_1, \frac{(\gamma - \varepsilon)^2}{4\alpha} \}$,
\begin{align*}
   - \left( \frac{\partial u}{\partial r} \right)^2 
   - \frac{\gamma - \varepsilon }{2r}\frac{\partial u}{\partial r}u 
   - \frac{\alpha }{r}|u|^2 
  \le 
   - \frac{1}{r} 
   \left\{ \alpha - \frac{ (\gamma - \varepsilon)^2 }{4r} \right\} |u|^2 
  \le 0,
\end{align*}
and moreover, $(3)$ implies that there exists a divergent sequence $\{t_i\}_{i=1}^{\infty}$ such that 
\begin{align*}
    \lim_{i \to \infty} \int_{S(t_i)} r^{ \gamma }
      \left\{ \left( \frac{\partial u}{\partial r} \right)^2 
      + \frac{1}{2}\lambda |u|^2
      + \frac{\gamma - \varepsilon }{2r} \frac{\partial u}{\partial r} u
      + \frac{\alpha }{r}|u|^2 
      \right\} \,dA = 0.
\end{align*}
Hence, substituting $t=t_i$ in $(7)$ and letting $ i \to \infty$, we get, for $s \ge r_2$, 
\begin{align*}
     \frac{1}{2} \int_{S(s)} r^{\gamma}
     \left\{ |\nabla u|^2 - \lambda |u|^2 \right\} \,dA 
     \ge 
     c_1 \int_{B(s,\infty)} r^{\gamma-1} 
     \left\{ |\nabla u|^2 + |u|^2 \right\} \,dv_g.
\end{align*}
Integrating this inequality with respect to $s$ over $[t,t_1]$ $(r_2 \le t < t_1)$, we have
\begin{align*}
   &  2 c_1 \int^{t_1}_t \,ds \int_{B(s,\infty)} r^{ \gamma - 1 }
      \left\{ |\nabla u|^2 + |u|^2 \right\} \,dv_g \\
 \le 
   &  \int_{B(t,t_1)} r^{\gamma}
      \left\{ |\nabla u|^2 - q|u|^2 \right\} \,dv_g \\
 = &  \left( \int_{S(t_1)} - \int_{S(t)} \right)
      r^{\gamma} \frac{\partial u}{\partial r}u \,dA
      - \gamma \int_{B(t,t_1)} r^{ \gamma - 1 }
      \frac{\partial u}{\partial r}u \,dv_g.
\end{align*}
In the last line, we have used the equation in Proposition $3.1$ with $\rho (r)=0$ and $ \psi = r^{\gamma} $. 
Since our assumption $(3)$ implies that 
\begin{align*}
    \liminf_{t_1\to \infty} \int_{S(t_1)} r^{\gamma}
    \frac{\partial u}{\partial r} u \,dA = 0,
\end{align*}
letting $t_1\to \infty$ and using Fubini's theorem, we have 
\begin{align}
   &  2c_1 \int^{\infty}_t \,ds \int_{B(s,\infty)} r^{\gamma-1}
      \left\{ |\nabla u|^2 + |u|^2 \right\} \,dv_g  \nonumber \\
 = &  2c_1 \int_{B(t,\infty)} (r-t) r^{\gamma-1}
      \left\{ |\nabla u|^2 + |u|^2 \right\} \,dv_g \nonumber \\
      \le 
   &  \int_{S(t)} r^{\gamma}
      \left\{ \left( \frac{\partial u}{\partial r} \right)^2 + |u|^2 
      \right\} \,dA
      + \gamma \int_{B(t,\infty)} r^{\gamma-1}
      \left\{ \left( \frac{\partial u}{\partial r} \right)^2 + |u|^2 \right\}
      \,dv_g < \infty,
\end{align}
where the right hand side of this inequality is finite by $(7)$. 
Hence we see that the desired assertion $(4)$ holds for $m = \gamma$. 

Once again, integrating this inequality $(8)$ with respect to $t$ over 
$[t_1,\infty)~(t_1\ge r_1)$ and using Fubini's theorem, we get
\begin{align*}
   & 2c_1 \int_{B(t,\infty)} (r-t)^2 r^{\gamma - 1} 
     \left\{ |\nabla u|^2 + |u|^ 2 \right\} \,dv_g \\
     \le 
   & \int_{B(t,\infty)} r^{\gamma}
     \left\{ \left( \frac{\partial u}{\partial r} \right)^2 + |u|^2 \right\} 
     \,dv_g
     + \gamma \int_{B(t,\infty)} (r-t) r^{\gamma - 1}
     \left\{ \left( \frac{\partial u}{\partial r} \right) ^2 + |u|^2 \right\} 
     \,dv_g \\
 < & \infty,
\end{align*}
where the right hand side of this inequality is finite by $(8)$. 
Thus, we see that the desired assertion $(4)$ holds for $ m = \gamma + 1 $. 
Repeating the integration with respect to $t$ shows that the assertion $(4)$ is valid for $m = \gamma + 2, \gamma + 3, \cdots $, therefore, for any $m > 0$. 
\end{proof}
\section{Exponential decay}
\begin{prop}
Under the assumption of Proposition $4.1$, we have 
\begin{equation*}
    \int_{B(r_0,\infty)} e^{ \eta r } 
    \left\{ |u|^2 + |\nabla u|^2 \right\} \,dv_g < \infty \qquad 
    {\rm for~any}~0 < \eta < \eta_1 (a,b,n) \sqrt{\lambda},
\end{equation*}
where we set
\begin{align*}
   \eta_1 (a,b,n) = 
   \begin{cases}
      2 \sqrt{ \frac{ (n+1)a - \widehat{b} }{ 2 - (n+1)a + \widehat{b} } } 
                    & \text{if} \quad 2 - ( n + 1 )a + \widehat{b} > 0, \\
      \quad \infty  & \text{if} \quad 2 - ( n + 1 )a + \widehat{b} \le 0 .
   \end{cases}
\end{align*}
\end{prop}
\begin{proof}
In Proposition $3.2$, let 
\begin{align*}
   & \rho (r) = m\log r ~~(m \ge \widehat{b}\,); \\
   & \gamma = 1; \\
   & \varepsilon = - ( \widehat{b} - 2a ).
\end{align*}
Then, 
\begin{align}
   & v = r^m u; \nonumber \\
   & q = \frac{m^2}{r^2} + \frac{m}{r^2} - \frac{m}{r}\Delta r + \lambda  
     \ge 
     \lambda 
     + \frac{m^2}{r^2} \left( 1 - \frac{ \widehat{b} - 1 }{m} \right) > 0 ; 
      \nonumber \\
   & r \frac{\partial q}{\partial r}
     = - \frac{2m^2}{r^2} - \frac{2m}{r^2} + \frac{m}{r}\Delta r 
     - m \frac{\partial (\Delta r)}{\partial r}  ; \nonumber \\
   & - \frac{\partial (\Delta r)}{\partial r}
     = |\nabla dr|^2 + \mathrm{Ric}\,(\nabla r,\nabla r)
 \ge \frac{\widehat{a} a}{r^2} - \frac{ \widehat{b}_1 (r) }{r};\tag{$*_9$}\\
   & r \Delta r + \varepsilon \ge  \widehat{a} + \varepsilon , \nonumber 
\end{align}
where we have used the identity $- \frac{\partial (\Delta r)}{\partial r} = |\nabla dr|^2 + \mathrm{Ric}\,(\nabla r,\nabla r)$ (see \cite{K3}, Proposition $2.3$). 
Hence, we have 
\begin{align*}
     & r \frac{\partial q}{\partial r} + q( r \Delta r + \varepsilon ) \\
 \ge & - \frac{2m^2}{r^2} - \frac{2m}{r^2} + \frac{m\widehat{a}}{r^2} 
       + \frac{m}{r^2} \widehat{a} a - \frac{m}{r} \widehat{b}_1 (r) 
       + \left( \lambda + \frac{m^2}{r^2} + \frac{m}{r^2} 
       - \frac{m\,\widehat{b}}{r^2} \right)( \widehat{a} + \varepsilon )  \\
  =  & ( \widehat{a} + \varepsilon ) \lambda 
       - \widehat{b}_1 (r) \frac{m}{r} - c_2 \frac{m}{r^2} 
       - c_3 \frac{m^2}{r^2},
\end{align*}
where we set 
\begin{align}
   &  c_2 = 2 - \widehat{a} ( 1 + a ) 
            + ( \widehat{b} - 1 )( \widehat{a} + \varepsilon ) ; \nonumber \\
   &  c_3 = 2 - \widehat{a} - \varepsilon  .
\end{align}
Note that on the Euclidean end ${\bf R}^n - B_{{\bf R}^n}(0,1)$, $a=b=1$, and hence, $c_2 = c_3 = 0$. 

Besides, 
\begin{align}
   & \widehat{a} + \varepsilon = \widehat{a} + 2a - \widehat{b} 
     = ( n - 1 )\left\{ \frac{n+1}{n-1} a - b \right\} > 0; \\
   & r ( \nabla dr )(\nabla v, \nabla v) 
     - \frac{1}{2} ( r\Delta r + \varepsilon ) \sum_{i=1}^{n-1} (dv(X_i))^2 
     \nonumber \\
   & \hspace{50mm} \ge  
     \frac{1}{2} \left( 2a - \widehat{b} - \varepsilon \right) 
     \sum_{i=1}^{n-1} (dv(X_i))^2 = 0; \nonumber \\
   & 1 - \frac{1}{2} ( r \Delta r + \varepsilon ) + 2m 
     \ge 1 - \frac{1}{2} ( \widehat{b} + \varepsilon ) + 2m  
     = 2m + 1 - a > 0 \nonumber 
\end{align}
for $m > m_0 := \frac{a-1}{2}$. 
Hence, we have 
\begin{align}
  &  \int_{S(t)} r 
     \left\{ \left( \frac{\partial v}{\partial r} \right)^2 
     + \frac{1}{2}q|v|^2 
     + \frac{ 1 - \varepsilon }{2r} \frac{\partial v}{\partial r}v 
     \right\} \,dA 
     + \frac{1}{2} \int_{S(s)} r 
     \bigl\{ |\nabla v|^2 - q|v|^2 \bigr\} \,dA  \nonumber \\
  &  - \int_{S(s)} r 
     \left\{ 
     \left( \frac{\partial v}{\partial r} \right)^2 
     + \frac{ 1 - \varepsilon }{2r} \frac{\partial v}{\partial r}v 
     \right\} \,dA  \nonumber \\
  \ge \, 
  &  \frac{1}{2} \int_{B(s,t)} 
     \left\{ ( \widehat{a} + \varepsilon ) \lambda 
     - \widehat{b}_1 (r) \frac{m}{r} 
     - \left( c_3 + \frac{c_2}{m} \right) 
     \frac{m^2}{r^2} 
     \right\} 
     |v|^2 \,dv_g \nonumber \\
  &  + ( 2m + 1 - a ) \int_{B(s,t)}  
     \left( \frac{\partial v}{\partial r} \right)^2 \,dv_g 
     + ( 1 - \varepsilon ) m \int_{B(s,t)}  
     \frac{1}{r} \frac{\partial v}{\partial r}v \,dv_g
\end{align}
for $m > m_0$. 
On the other hand, Lemma $3.1$ with $ \beta = 0 $ yields that
\begin{align}
      \left( \int_{S(t)} - \int_{S(s)} \right) |v|^2 \,dA
  = & \int_{B(s,t)}  
      \left\{ \left( \Delta r \right) |v|^2
      + 2v\frac{\partial v}{\partial r} \right\} \,dv_g \nonumber \\
  \ge 
    & \int_{B(s,t)}  
      \left\{ \frac{ \widehat{a} }{r} |v|^2
      + 2v\frac{\partial v}{\partial r} \right\} \,dv_g 
\end{align}
Multiplying the inequality $(12)$ by a constant $\alpha > 0$ and addition of it to $(11)$ make 
\begin{align}
  &  \int_{S(t)} r \left\{ \left( \frac{\partial v}{\partial r} \right)^2 
     + \frac{1}{2}q|v|^2 
     + \frac{ 1 - \varepsilon }{2r} \frac{\partial v}{\partial r}v 
     + \frac{\alpha}{r}|v|^2 \right\} \,dA  \nonumber \\
  &  + \frac{1}{2} \int_{S(s)} r 
     \bigl\{ |\nabla v|^2 - q|v|^2 \bigr\} \,dA
     - \int_{S(s)} r 
     \left\{ \left( \frac{\partial v}{\partial r} \right)^2 
     + \frac{ 1 - \varepsilon }{2r} \frac{\partial v}{\partial r}v 
     + \frac{ \alpha }{r}|v|^2  
     \right\} \,dA  \nonumber \\
  \ge \, 
  &  \frac{1}{2} \int_{B(s,t)} 
     \left\{ ( \widehat{a} + \varepsilon ) \lambda  
     + \frac{ 2 \alpha \widehat{a} }{r} - \widehat{b}_1 (r) \frac{m}{r} 
     - \left( c_3 + \frac{c_2}{m} \right) 
     \frac{m^2}{r^2} 
     \right\} |v|^2 \,dv_g  \nonumber \\
  &  + ( 2m + 1 - a ) \int_{B(s,t)} 
     \left( \frac{\partial v}{\partial r} \right)^2 \,dv_g 
     + \int_{B(s,t)} 
     \left\{ 2 \alpha + \frac{ ( 1 - \varepsilon ) m }{r} \right\}
     \frac{\partial v}{\partial r} v \,dv_g 
\end{align}
for $m \ge m_0$. 
Substituting the inequality
\begin{align*}
   &  \left\{ 2 \alpha + \frac{ ( 1 - \varepsilon ) m }{r} \right\} 
      \frac{\partial v}{\partial r} v  \\
   \ge 
   &  - ( 2m + 1 - a ) \left( \frac{\partial v}{\partial r} \right) ^2 
      - \frac{1}{ 4( 2m + 1 - a ) } 
      \left\{ 2 \alpha + \frac{ ( 1- \varepsilon ) m }{r} \right\}^2 |v|^2 ,
\end{align*}
into $(13)$, we get 
\begin{align}
 & \int_{S(t)} r 
   \left\{ \left( \frac{\partial v}{\partial r} \right)^2 + \frac{1}{2}q|v|^2 
   + \frac{ 1 - \varepsilon }{2r} \frac{\partial v}{\partial r}v
   + \frac{\alpha}{r}|v|^2 
   \right\} \,dA  \nonumber \\
 & + \frac{1}{2} \int_{S(s)} r \bigl\{ |\nabla v|^2 - q|v|^2 \bigr\} \,dA
   - \int_{S(s)} r \left\{ \left( \frac{\partial v}{\partial r} \right)^2 
   + \frac{ 1 - \varepsilon }{2r} \frac{\partial v}{\partial r}v 
   + \frac{ \alpha }{r}|v|^2 \right\} \,dA  \nonumber \\
  \ge \, 
 & \int_{B(s,t)} H( \alpha, r, m ) \,|v|^2 \,dv_g ,
\end{align}
where we set 
\begin{align*}
  & H( \alpha, r, m ) \nonumber \\
  = 
  & ( \widehat{a} + \varepsilon ) \lambda  
    - \frac{1}{ 4( 2m + 1 - a ) } 
    \left\{ 2 \alpha + \frac{ ( 1- \varepsilon ) m }{r} \right\}^2 
    + \frac{ 2 \alpha \widehat{a} }{r} - \widehat{b}_1 (r) \frac{m}{r} 
    - \left( c_3 + \frac{c_2}{m} \right) \frac{m^2}{r^2} \nonumber \\
  = 
  & ( \widehat{a} + \varepsilon ) \lambda 
    - \frac{ \alpha^2 }{ 2m + 1 -a } 
    + \left\{ 2 \widehat{a} 
    - \frac{ 1 - \varepsilon }{ 2 + \frac{ 1 - a }{m} } \right\} 
    \frac{\alpha}{r} - \widehat{b}_1 (r) \frac{m}{r} 
    - \left( c_3 + \frac{c_4}{m} \right) \frac{m^2}{r^2}
\end{align*}
and 
\begin{align*}
   c_4 
   = c_2 + \frac{ ( 1 - \varepsilon )^2 }{4( 2 + \frac{ 1 - a }{m})} 
   \le 
   c_2 + \frac{ ( 1 - \varepsilon )^2 }{8}
\end{align*}
Therefore, for any $\theta \in (0,1)$, if we take $\alpha _1 = \alpha _1( \lambda, a,b,n,r_0,\theta ) >0 $ sufficiently small, then for $0 < \alpha < \alpha _1 $, the right hand side of $(14)$ is bounded from below by 
\begin{align*}
  \int_{B(s,t)}  
  \left\{ ( 1 - \theta ) ( \widehat{a} + \varepsilon ) \lambda 
  - \widehat{b}_1 (r) \frac{m}{r} 
  - \left( c_3 + \frac{c_4}{m} \right)  \frac{m^2}{r^2}
  \right\}|v|^2 \,dv_g.
\end{align*}
Our taking $\lim_{r \to \infty}q = \lambda$ into account, Proposition $4.1$ implies that $|\nabla v|$ and $v$ are in $L^2\bigl( B(r_0,\infty),dv_g \bigr)$, and hence, 
\begin{align*}
  \liminf_{t\to \infty} \int_{S(t)} r 
  \left\{ \left( \frac{\partial v}{\partial r} \right)^2 
  + \frac{ 1 - \varepsilon }{2r} \frac{\partial v}{\partial r}v 
  + \frac{1}{2} q |v|^2 + \frac{\alpha}{r}|v|^2 \right\} \,dA = 0.
\end{align*}
Therefore, substituting appropriate divergent sequence $\{t_i\}$ for $t$ in $(14)$, and letting $t_i\to \infty$, we get
\begin{align}
 & \int_{S(s)} r \left\{ |\nabla v|^2 - q|v|^2 \right\} \,dA
   - 2 \int_{S(s)} r \left\{ \left( \frac{\partial v}{\partial r} \right)^2 
   + \frac{ 1 - \varepsilon }{2r} \frac{\partial v}{\partial r}v
   + \frac{\alpha }{r} |v|^2 \right\} \,dA  \nonumber \\
 \ge 
 & 2 \int_{B(s,\infty)} 
   \left\{ ( 1 - \theta ) ( \widehat{a} + \varepsilon ) \lambda 
   - \widehat{b}_1 (r) \frac{m}{r} 
   - \left( c_3 + \frac{c_4}{m} \right) \frac{m^2}{r^2}
   \right\}|v|^2 \,dv_g.
\end{align}
If our end $E$ equals a Euclidean end ${\bf R}^n - B_{{\bf R}^n}(0,r_1)$, then $c_3=0$. 
However, in our general situation, the sign of this constant $c_3$ may be negative. Hence, we shall set 
\begin{align}
   \overline{c}_3 = \max\{ c_3, \theta \} ; \quad 
   \overline{c}_4 = \max\{ c_4, 1 \}.  
\end{align}
Then, for $ m \ge m_1(c_3,c_4,\theta ) ~(~> \max \{\widehat{b} + 1 , m_0\})$, 
\begin{align*}
   \overline{c}_3 + \frac{\overline{c}_4}{m} 
   \le ( 1 + \theta ) \overline{c}_3
\end{align*}
Multiplying both side of $(15)$ by $ s^{ - 2m } $ and integrating it with respect to $s$ over $[x,\infty)$ $(x>r_0)$, we have
\begin{align}
 &  \int_{B(x,\infty)} r^{ 1 - 2m } 
    \left\{ |\nabla v|^2 - q|v|^2 \right\} \,dv_g  \nonumber \\
 &  - 2 \int_{B(x,\infty)} r^{ 1 - 2m }
    \left\{ \left( \frac{\partial v}{\partial r} \right)^2 
    + \frac{ 1 - \varepsilon }{2r} \frac{\partial v}{\partial r}v 
    + \frac{\alpha }{r}|v|^2 \right\} \,dv_g \nonumber \\ 
  \ge 
 &  2 \int_{x}^{\infty} s^{ - 2m } \,ds 
    \int_{B(s,\infty)} 
    \left\{ ( 1 - \theta ) ( \,\widehat{a} + \varepsilon ) \lambda 
    - \widehat{b}_1 (r) \frac{m}{r} 
    - ( 1 + \theta ) \overline{c}_3 \frac{m^2}{r^2} \right\}
    |v|^2 \,dv_g  \nonumber \\
  \ge 
 &  2 \int_{x}^{\infty}
    \left\{ ( 1 - \theta ) ( \,\widehat{a} + \varepsilon ) \lambda 
    - \widehat{b}_1 (s) \frac{m}{s} 
    - ( 1 + \theta ) \overline{c}_3 \frac{m^2}{s^2} \right\} 
    s^{-2m} \,ds 
    \int_{B(s,\infty)}|v|^2 \,dv_g \nonumber \\
  \ge 
 &  2 \left\{ ( 1 - \theta ) ( \,\widehat{a} + \varepsilon ) \lambda 
    - \widehat{b}_1 (x) \frac{m}{x} 
    - ( 1 + \theta ) \overline{c}_3 \frac{m^2}{x^2} \right\} 
    \int_{x}^{\infty} s^{-2m} \,ds
    \int_{B(s,\infty)} |v|^2 \,dv_g
\end{align}
for $m \ge m_1$. 
Substitution of the equation in Lemma $3.2$ into $(17)$ makes 
\begin{align*}
 & - \frac{1}{2} \frac{d}{dx} \left( x^{ 1 - 2m } \int_{S(x)}|v|^2 
    \,dA\right)
    - \frac{1}{2} \int_{S(x)} r^{-2m} 
    \left\{ 2m - 1 - r \Delta r \right\} |v|^2 \,dA   \nonumber \\
 &  - \int_{B(x,\infty)} r^{1-2m} 
    \left\{ 2 \left( \frac{\partial v}{\partial r} \right)^2 
    + \frac{ 2 - \varepsilon }{r} \frac{\partial v}{\partial r}v 
    + 2 \frac{\alpha }{r}|v|^2 \right\} \,dv_g   \nonumber \\
 \ge 
 &  2 \left\{ ( 1 - \theta ) ( \,\widehat{a} + \varepsilon ) \lambda 
    - \widehat{b}_1 (x) \frac{m}{x} 
    - ( 1 + \theta ) \overline{c}_3 \frac{m^2}{x^2} \right\} 
    \int_{x}^{\infty} s^{-2m} \,ds \int_{B(s,\infty)} |v|^2 \,dv_g.
\end{align*}
Taking the following three inequalities into account 
\begin{align*}
 & 2 \left( \frac{\partial v}{\partial r} \right)^2 + 
   \frac{ 2 - \varepsilon }{r} \frac{\partial v}{\partial r} v 
   + 2 \frac{\alpha }{r} |v|^2 
 \ge  
   \frac{1}{r} 
   \left\{ 2 \alpha - \frac{ ( 2 - \varepsilon )^2 }{8r} \right\} |v|^2
   \ge 0 \\
 & \hspace{70mm} 
   {\rm for}~~r > r_1 = \frac{( 2 - \varepsilon )^2}{16 \alpha} ; \\
 & 2m - 1 - r \Delta r \ge 2m - \widehat{b} - 1 
   \ge 2 ( 1 - \theta ) m 
   \qquad {\rm for}~~m \ge m_2=\frac{ \widehat{b} + 1 }{\theta}; \\
 & \theta > \widehat{b}_1 (r) \,(>0) 
   \qquad {\rm for}~~r \ge r_2 = r_2(b_1,\theta), 
\end{align*}
we have for any $ m > m_3 = \max\{ m_1, m_2 \} $ and $ x > r_3 := \max\{ r_1, r_2 \} $
\begin{align}
  & - \frac{1}{2} \frac{d}{dx} \left (x^{1-2m} \int_{S(x)} |v|^2 \,dA \right)
    - \frac{ ( 1 - \theta ) m }{x} \left( x^{ 1 - 2m } \int_{S(x)} |v|^2 
    \,dA \right) \nonumber \\
 \ge 
  & 2 \left\{ ( 1 - \theta ) ( \widehat{a} + \varepsilon ) \lambda 
    - \theta \frac{m}{x}
    - ( 1 + \theta ) \overline{c}_3 \frac{m^2}{x^2} \right\} 
    \int_{x}^{\infty} s^{-2m} \,ds \int_{B(s,\infty)} |v|^2 \,dv_g .
\end{align}
Now, for $ m > m_3 $ and $ x > r_3 $, we shall set 
\begin{align}
   \frac{m}{x} 
 = \frac{ - \theta + \sqrt{ \theta ^2 + 4 ( 1 + \theta ) \overline{c}_3 ( 1 - \theta ) ( \widehat{a} + \varepsilon ) \lambda } }{ 2 ( 1 + \theta ) \overline{c}_3 } =: \overline{c}_6 >0
\end{align}
and 
\begin{align*} 
   F(x) = x^{1-2m} \int_{S(x)} |v|^2 \, dA
        = x \int_{S(x)} |u|^2 \,dA .
\end{align*}
Then, 
\begin{align*}
   ( 1 - \theta ) ( \widehat{a} + \varepsilon ) \lambda 
   - \theta \,\overline{c}_6 
   - ( 1 + \theta ) \overline{c}_3 \,\overline{c}_6^2 = 0,
\end{align*}
and hence, the inequality $(18)$ implies that 
\begin{align} 
    F'(x) + 2 ( 1 - \theta ) \overline{c}_6 F(x)\le 0 
    \qquad {\rm for}~~x \ge r_3.
\end{align}
If we set $ G(x) = e^{ 2 ( 1 - \theta ) \overline{c}_6 \,x} F(x) $, $(20)$ reduces to 
\begin{align*}
   G(x)' \le 0  \qquad {\rm for}~~x \ge r_3.
\end{align*}
Thus, 
$ G(x) \le G(r_2)$ for $x \ge r_3$, that is, 
\begin{align} 
   F(x) = x \int_{S(x)} |u|^2 \,dA 
   \le 
   e^{ - 2 ( 1 - \theta ) \overline{c}_6 x } F(r_2) e^{ \overline{c}_6 r_2 }
   \qquad {\rm for}~~x \ge r_3 .
\end{align}
Now, in view of $(9)$, $(10)$, $(16)$, and $(19)$, we see that $(21)$ implies that 
\begin{align}
   \int_{B(r_0,\infty)} e^{ \eta r } |u|^2 \,dv_g < \infty 
   \qquad {\rm for~any}~~0 < \eta < \eta_1 (a,b,n) \sqrt{\lambda}.
\end{align}

Next, we shall show that Proposition $3.1$ and $(22)$ yield 
\begin{align*}
   \int_{B(r_0,\infty)} e^{ \eta r } |\nabla u|^2 \,dv_g < \infty 
   \qquad {\rm for~any}~~0 < \eta < \eta_1 (a,b,n) \sqrt{\lambda}.
\end{align*}
For that purpose, first consider the integral
\begin{align*}
    g(R)
    = 2 \int_{B(r_0,R)} e^{ \eta r } u \frac{\partial u}{\partial r} \,dv_g .
\end{align*} 
Then, Green's formula yields
\begin{align*}
      g(R)
  = & \frac{1}{\eta} \int_{B(r_0,R)}
      \left\langle 
      \nabla \left( e^{ \eta r } \right),\nabla \left( u^2 \right) 
      \right\rangle \,dv_g \\
  = & \frac{1}{\eta} \left( \int_{S(R)} - \int_{S(r_0)} \right)
      e^{ \eta r} |u|^2 \,dA 
      - \int_{B(r_0,R)} ( \Delta r + \eta ) e^{ \eta r }|u|^2 \,dv_g.
\end{align*}
Since $\lim_{r \to \infty } \Delta r = 0$, $(22)$ implies the existence of the limit, $\lim_{ R \to \infty }g(R)$. 
In particular, 
\begin{equation}
   \liminf_{R\to \infty } e^{ \eta R}
   \left| \int_{S(R)} u \frac{\partial u}{\partial r} \,dA \right| = 0.
\end{equation}
In Proposition $3.1$, we put $\rho = 0$ and $\psi = e^{ \eta r}$. 
Then $v = u$, $ q = \lambda $, and 
\begin{align*}
    &  \int_{B(r_0,R)} \left\{ |\nabla u|^2 - \lambda |u|^2 \right\} 
       e^{ \eta r} \,dv_g \\
 =  &  \left( \int_{S(R)} - \int_{S(r_0)} \right)
       \frac{\partial u}{\partial r} u e^{ \eta r} \,dA
       -  \eta \int_{B(r_0,R)}
       e^{ \eta r} \frac{\partial u}{\partial r} u \,dv_g \\
\le & \left( \int_{S(R)} - \int_{S(r_0)} \right)
      \frac{\partial u}{\partial r} u e^{ \eta r} \,dA
      + \frac{1}{2} \int_{B(r_0,R)}
      e^{ \eta r} \left\{ |\nabla u|^2 +  \eta ^2 |u|^2 \right\} \,dv_g.
\end{align*}
Hence,
\begin{align*}
     & \frac{1}{2} \int_{B(r_0,R)} e^{ \eta r} |\nabla u|^2 \,dv_g \\
 \le & \left( \int_{S(R)} - \int_{S(r_0)} \right)
       \frac{\partial u}{\partial r} u e^{ \eta r} \,dA
       + \int_{B(r_0,R)}
       \left\{ \frac{ \eta ^2}{2} + \lambda \right\} e^{ \eta r} |u|^2 \,dv_g.
\end{align*}
Therefore, $(22)$ and $(23)$ imply that
\begin{equation}
     \int_{B(r_0,\infty)} e^{ \eta r} |\nabla u|^2 \,dv_g < \infty 
     \qquad {\rm for~any}~~0 < \eta < \eta_1 (a,b,n) \sqrt{\lambda} . 
\end{equation}
Thus, from $(22)$ and $(24)$, we get our desired result. 
\end{proof}
\section{Vanishing on some neighborhood of infinity}
\begin{prop}
Under the assumptions of Proposition $3.1$, $u\equiv 0$ on $B(r_0,\infty)$. 
\end{prop}
\begin{proof}
In proposition $3.2$, we shall put 
\begin{align*}
   \gamma  & = 1 ; \\
   \rho(r) & = k r^{\theta }\quad ( k\ge 1,~\theta \in (0,1) )   
\end{align*}
and choose 
$ \varepsilon $ so that $(6)$ holds. 
Then, from
\begin{align}
   &  - \frac{\partial (\Delta r)}{\partial r}
      = |\nabla dr|^2 + \mathrm{Ric}\,(\nabla r,\nabla r)
 \ge  \frac{\widehat{a} a}{r^2} - \frac{ \widehat{b}_1 (r) }{r};\tag{$*_9$} \\
   &  \hspace{10mm} \widehat{a} \le r\Delta r \le \widehat{b}, \nonumber
\end{align}
we get
\begin{align}
 v & = e^{kr^{\theta }} u ;\\
 q & = \lambda - \rho''(r) - \rho'(r) \Delta r + \left( \rho'(r) \right)^2 
       \nonumber \\
   & = \lambda + k \theta ( 1 - \theta ) r^{ \theta - 2 }
       - k \theta r^{\theta -1} \Delta r + k^2 \theta ^2 r^{ 2\theta - 2 } \\
   & \ge 
       \lambda 
       + k \theta \left\{ 1 - \theta - \widehat{b} \right\} r^{ \theta - 2 } 
       + k^2 \theta ^2 r^{ 2\theta - 2 }; \nonumber \\
 r \frac{\partial q}{\partial r} 
   & = - k \theta ( 1 - \theta )( 2 - \theta ) r^{\theta - 2} 
       + k \theta ( 1 - \theta ) r^{\theta -1} \Delta r  \nonumber \\
   & \hspace{3.5mm}  
       - k \theta r^{\theta } \frac{\partial (\Delta r)}{\partial r} 
       - 2 k^2 \theta ^2 ( 1 - \theta ) r^{ 2\theta - 2 }  \nonumber \\
   & \ge 
       - k \theta ( 1 - \theta )( 2 - \theta ) r^{ \theta - 2} 
       + k \theta ( 1 - \theta ) \widehat{a} r^{ \theta - 2}  \nonumber \\
   & \hspace{3.5mm}  
       + k \theta \widehat{a} a r^{ \theta - 2} 
       - k \theta r^{ \theta - 1 } \widehat{b}_1(r) 
       - 2 k^2 \theta ^2 ( 1 - \theta ) r^{ 2 \theta - 2 }  \nonumber \\
   & = - k \theta r^{ \theta - 1 } \widehat{b}_1(r) 
      + k \theta \left\{ 
      - ( 1 - \theta )( 2 - \theta -\widehat{a} ) + \widehat{a}a 
      \right\} r^{ \theta - 2}
       - 2 k^2 \theta ^2 ( 1 - \theta ) r^{ 2 \theta - 2 }.\nonumber 
\end{align}
In addition, 
\begin{align*}
    r \Delta r + \varepsilon \ge \widehat{a} + \varepsilon > 0 .
\end{align*}
Therefore, 
\begin{align}
 & r \frac{\partial q}{\partial r} + q ( r \Delta r + \varepsilon ) 
   \nonumber \\
 \ge 
 & - k \theta r^{ \theta - 1 } \widehat{b}_1(r)
   + k \theta 
   \left\{ - ( 1 - \theta )( 2 - \theta -\widehat{a} ) + \widehat{a}a \right\} 
   r^{ \theta - 2} 
   - 2 k^2 \theta ^2 ( 1 - \theta ) r^{ 2 \theta - 2 } \nonumber \\
 & + \left\{ \lambda 
   + k \theta ( 1 - \theta - \widehat{b} ) r^{ \theta - 2 } 
   + k^2 \theta ^2 r^{ 2\theta - 2 } \right\} 
   \left( \widehat{a} + \varepsilon \right) \nonumber \\
=& \left( \widehat{a} + \varepsilon \right) \lambda 
   - k \theta r^{ \theta - 1 } \widehat{b}_1(r) 
   + k \theta r^{ \theta - 2}
   \left\{ ( 1 - \theta ) ( - 2 + \theta + 2 \widehat{a} + \varepsilon ) 
   + \widehat{a}a - \widehat{b} ( \widehat{a} + \varepsilon ) \right\} 
   \nonumber \\
 & + k^2 \theta ^2 r^{ 2\theta - 2 } 
   \left\{ \widehat{a} + \varepsilon - 2 ( 1 - \theta ) \right\} \nonumber \\
=& \left( \widehat{a} + \varepsilon \right) \lambda 
   - k \theta r^{ \theta - 1 } \widehat{b}_1(r)
   + k \theta r^{ \theta - 2} c_7 + k^2 \theta ^2 r^{ 2\theta - 2 } c_8
\end{align}
and 
\begin{align}
   & 1 - \frac{1}{2} ( r \Delta r + \varepsilon ) + 2 r \rho '(r) 
     \ge
     2 k \theta r^{\theta} + 1 - \frac{1}{2}( \widehat{b} + \varepsilon ) 
     = 2 k \theta r^{\theta} + c_9 ; \\
   & ( 1 - \varepsilon ) k \theta r^{ \theta -1 } 
     \frac{\partial v}{\partial r}v 
     \ge 
     - \left\{ 2 k \theta r^{\theta} + c_9 \right\} 
     \left( \frac{\partial v}{\partial r} \right)^2
     - \frac{ ( 1 - \varepsilon )^2 k^2 \theta^2 r^{ 2\theta - 2 } }
     {2 k \theta r^{\theta} + c_9 } |v|^2 ,
\end{align}
where we set 
\begin{align*}
  c_7 & = ( 1 - \theta ) ( - 2 + \theta + 2 \widehat{a} + \varepsilon ) 
        + \widehat{a}a - \widehat{b} ( \widehat{a} + \varepsilon );\\
  c_8 & = \widehat{a} + \varepsilon - 2 ( 1 - \theta ); \\
  c_9 & = 1 - \frac{1}{2}( \widehat{b} + \varepsilon )
\end{align*}
for simplicity. 
Since $\widehat{a} - \varepsilon > 0$, we can choose $ \theta \in (0,1) $ so that 
\begin{align}
     c_8 = \widehat{a} - \varepsilon - 2( 1 - \theta ) > 0 . 
\end{align}
Then, from $(27)$, $(28)$, and $(29)$, we see that 
\begin{align}
  &  \frac{1}{2} \left\{ r \frac{\partial q}{\partial r} 
     + q ( r \Delta + \varepsilon ) \right\} |v|^2
     + ( 1 - \varepsilon ) \rho '(r) \frac{\partial v}{\partial r} v  
     \nonumber \\
  &  \hspace{30mm} + \left\{ 1 - \frac{1}{2} ( r \Delta r + \varepsilon ) 
     + 2 r \rho '(r) 
     \right\} \left( \frac{\partial v}{\partial r} \right)^2  \nonumber \\
 \ge 
  &  \frac{1}{2} \left\{ ( \widehat{a} + \varepsilon ) \lambda 
     - k \theta r^{ \theta - 1 } \widehat{b}_1(r) 
     + k^2 \theta ^2 r^{ 2\theta - 2 }
     \left( c_8 + \frac{c_7}{k\theta r^{\theta}} 
     - \frac{ ( 1 - \varepsilon )^2 }{ 2( 2k\theta r^{\theta } + c_9 ) } 
     \right) \right\} |v|^2.
\end{align}
Thus, in view of $(30)$, there exists a constant $ r_6(a,b,\varepsilon, \theta) > r_0 $ depending only on $a,b,\varepsilon$, and $\theta$ such that 
\begin{align*}
   c_8 + \frac{c_7}{k\theta r^{\theta}} 
   - \frac{ ( 1 - \varepsilon )^2 }{ 2( 2k\theta r^{\theta } + c_9 ) } 
   \ge \frac{ c_8 }{2} := c_9
\end{align*}
for any $k\ge 1$ and $r\ge r_6$. 
Since $\lim_{r \to \infty} b_1(r) = 0$, we have
\begin{align*}
   ( \widehat{a} + \varepsilon ) \lambda - y \widehat{b}_1(r) 
   + y^2 c_9  
  = c_9 \left( y - \frac{\widehat{b}_1(r)}{2c_9} \right)^2 
   + \frac{ 4 ( \widehat{a} + \varepsilon ) \lambda c_9 - \widehat{b}_1(r) ^2 }
   {4c_9} > 0
\end{align*}
for any $r \ge r_7(a,\lambda, \varepsilon)$ and $y \in {\bf R}$. 
Therefore, the right hand side of $(31)$ is nonnegative for any $k\ge 1$ and $r\ge r_8:=\max\{ r_6, r_7 \}$ 
Thus, we have for any $k\ge 1$ and $t > s \ge r_8$
\begin{align}
  & \int_{S(t)} r \left\{ 
    \left( \frac{\partial v}{\partial r} \right)^2 + \frac{1}{2}q|v|^2 
    - \frac{1}{2}|\nabla v|^2 
    + \frac{ 1 - \varepsilon }{2r}
    \frac{\partial v}{\partial r}v \right\} \,dA \nonumber \\
  & + \int_{S(s)} r \left\{ \frac{1}{2}|\nabla v|^2
    - \frac{1}{2}q|v|^2 - \left(\frac{\partial v}{\partial r}\right)^2
    - \frac{ 1 - \varepsilon }{2r}\frac{\partial v}{\partial r}v \right\} 
    \,dA
  \ge 0.
\end{align}
By bearing $(25)$ in mind, we see that Proposition $5.1$ implies that 
\begin{align*}
    \liminf_{t\to \infty} \int_{S(t)} r 
    \left\{ |\nabla v|^2 + |v|^2 \right\} \,dA = 0.
\end{align*}
Hence, substituting an appropriate divergent sequence $\{t_i\}$ for $t$ in $(32)$, and letting $t_i\to \infty$, we see that 
\begin{align}
    \int_{S(s)} \left\{ - \left( \frac{\partial v}{\partial r} \right)^2 
    + \frac{1}{2} |\nabla v|^2 - \frac{1}{2} q|v|^2 
    - \frac{ 1 - \varepsilon }{2r}\frac{\partial v}{\partial r} v 
    \right\} \,dA \ge 0
\end{align}
for all $ s \ge r_8$ and $ k \ge 1$. 
On account of the facts
\begin{align*}
  &  \left( \frac{\partial v}{\partial r} \right)^2
     = \left\{ k^2 \theta^2 r^{ 2 \theta - 2 } |u|^2
     + 2k \theta r^{\theta -1} \frac{\partial u}{\partial r} u
     + \left( \frac{\partial u}{\partial r} \right)^2 \right\} 
     e^{2kr^{\theta}} , \\
  &  |\nabla v|^2
     = \left\{ k^2 \theta^2 r^{2\theta -2} |u|^2
     + 2k \theta r^{\theta -1} \frac{\partial u}{\partial r}u
     + |\nabla u|^2 \right\} e^{2kr^{\theta}},
\end{align*}
and $(26)$, the left hand side of $(32)$ is written as follows:
\begin{align*}
   \left\{ k^2 I_1(s) + k I_2(s) + I_3(s) \right\} e^{2kr^{\theta}},
\end{align*}
where 
\begin{align*}
   I_1(s) = - \theta^2 s^{ 2 \theta - 2 } \int_{S(s)} |u|^2 \,dA;
\end{align*}
$I_2(s)$ and $I_3(s)$ is independent of $k$. 
Thus, for any fixed $ s\ge r_8$, the inequality $ k^2 I_1(s) + k I_2(s) + I_3(s) \ge 0 $ holds for all $ k \ge 1$. 
Therefore, $ I_1(s) = 0 $ for any fixed $s\ge r_8$, that is, $u \equiv 0$ on $B( r_8, \infty)$. 
The unique continuation theorem implies that $ u \equiv 0$ on $E = M - \overline{U}$. 
\end{proof}
 
We obtain Theorem $1.1$ from this Proposition $6.1$. \vspace{3mm}

\noindent {\it Proof of Theorem $1.2$}

Under the assumptions of Theorem $1.2$, $\lim_{r \to \infty} \Delta = 0$, and hence, $\sigma_{{\rm ess}}( - \Delta ) =[0,\infty)$ (see \cite{K1}). 
If $\lambda >0$ is an eigenvalue of $ - \Delta $ and $u$ is an corresponding eigenfunction, then $u,\nabla u \in L^2(M,dv_g)$, in particular, $u,\nabla u \in L^2(E,dv_g)$. 
However, $(*_5)$ and Theorem $1.1$ implies that $(1)$ with $\gamma =1$ holds:
\begin{align*}
   \liminf_{t \to \infty}~t \int_{S(t)} 
   \left\{ \left( \frac{\partial u}{\partial r} \right)^2 + |u|^2 \right\}
   \,dA \neq 0 
\end{align*}
Therefore, there exist positive constants $c_{10}$ and $r_9$ such that 
\begin{align*}
   t \int_{S(t)} 
   \left\{ \left( \frac{\partial u}{\partial r} \right)^2 + |u|^2 \right\}
   \,dA > c_{10} \qquad {\rm for}~t \ge r_9.
\end{align*}
Hence, dividing the both sides of this inequality by $t$ and integrating it with respect to $t$ over $[r_9,\infty)$, we get
\begin{align*}
   \int_{B(r_9,\infty)} 
   \left\{ \left( \frac{\partial u}{\partial r} \right)^2 + |u|^2 \right\}
   \,dA > \int_{r_9}^{\infty} \frac{c_{10}}{t} = + \infty.
\end{align*}
This contradicts the fact that $u,\nabla u \in L^2(E,dv_g)$. 
Hence, $\sigma_{{\rm p}}( - \Delta ) = \emptyset$. 
Thus, we have proved Theorem $1.2$. \vspace{3mm}

Theorem $1.3$ and $1.4$ are obtained by using the comparison theorem in Riemannian geometry, that is, Proposition $2.1$ and Proposition $2.2$.

\section{Remarks}

In our theorems, we assume that there exists an open subset $U$ of $M$ with compact boundary $\partial U$ such that the outward pointing normal exponential map $ \exp_{\partial U} : N^{+}(\partial U) \to M - \overline{U}$ induces a diffeomorphism. 
This condition is not essential. 
What matters is rather the existence of a function with special properties, such as $r$. 
The readers interested in this matter could pick up necessary conditions that should be satisfied by such a function from our proof above. 
We note that there are Donnelly's works (\cite{D4},\cite{D5}) from this viewpoint of an exhaustion function of $M$. 

Our arguments are also applicable for the case that the metric of an end is a warped product. 
This case is discussed in \cite{K4}.


\vspace{1mm}
\begin{flushleft}
Hironori Kumura\\ 
Department of Mathematics\\ 
Shizuoka University\\ 
Ohya, Shizuoka 422-8529\\ 
Japan\\
E-mail address: smhkumu@ipc.shizuoka.ac.jp
\end{flushleft}

\end{document}